\documentclass[11pt]{article}%
\usepackage{amsmath}
\usepackage{amsfonts}
\usepackage{amssymb}
\usepackage{graphicx}%
\newtheorem{theorem}{Theorem}[section]

\newtheorem{corollary}[theorem]{Corollary}

\newtheorem{example}[theorem]{Example}

\newtheorem{lemma}[theorem]{Lemma}

\newtheorem{proposition}[theorem]{Proposition}
\newtheorem{remark}[theorem]{Remark}

\newenvironment{proof}[1][Proof]{\noindent\textbf{#1.} }{\ \rule{0.5em}{0.5em}}
\begin{document}

\title{Orthosymmetric spaces over an Archimedean vector lattice}
\author{M. A. Ben Amor, K. Boulabiar, and J. Jaber\\{\small Research Laboratory of Algebra, Topology, Arithmetic, and Order}\\{\small Department of Mathematics}\\{\small Faculty of Mathematical, Physical and Natural Sciences of Tunis}\\{\small Tunis-El Manar University, 2092-El Manar, Tunisia}}
\date{}
\maketitle

\begin{abstract}
We introduce and study the notion of orthosymmetric spaces over an Archimedean
vector lattice as a generalization of finite-dimentional Euclidean inner
spaces. A special attention has been paid to linear operators on these spaces.

\end{abstract}

\section{Introduction and first properties}

We take it for granted that the reader is familiar with the elementary theory
of vector lattices (i.e., Riesz spaces) and positive operators. For
terminology, notation, and properties not explained or proved we refer to the
standard texts \cite{LZ71,Z83}.

Let $\mathbb{V}$ be an Archimedean vector lattice. Following Buskes and van
Rooij in \cite{BR00}, we call a $\mathbb{V}$-\textsl{valued orthosymmetric
product} on a vector lattice $L$ any function that takes each ordered pair
$\left(  f,g\right)  $ of elements of $L$ to a vector $\left\langle
f,g\right\rangle $ of $\mathbb{V}$ and has the two following properties.

\begin{enumerate}
\item[(1)] (\textsl{Positivity}) $\left\langle f,g\right\rangle \in
\mathbb{V}^{+}$ for all $f,g\in L^{+}$.

\item[(2)] (\textsl{Orthosymmetry}) $\left\langle f,g\right\rangle =0$ in
$\mathbb{V}$ for all $f,g\in L$ with $f\wedge g=0$.
\end{enumerate}

\noindent By an \textsl{orthosymmetric space over }$\mathbb{V}$ (or, just an
\textsl{orthosymmetric space }if no confusion can arise) we mean a vector
lattice $L$ along with a $\mathbb{V}$-valued orthosymmetric product on $L$. As
the next example shows, the classical Euclidean spaces fill within the
framework of orthosymmetric spaces.

\begin{example}
\label{Euclidean}As usual, $\mathbb{R}$ denotes the Archimedean vector lattice
of all real numbers. Pick $n\in\mathbb{N}=\left\{  1,2,...\right\}  $ and
suppose that the vector space $\mathbb{R}^{n}$ is equipped with its usual
structure of Euclidean space. In particular,%
\[
\left\langle f,g\right\rangle =\sum_{k=1}^{n}\left\langle f,e_{k}\right\rangle
\left\langle g,e_{k}\right\rangle \text{ for all }f,g\in\mathbb{R}^{n},
\]
where $\left(  e_{1},...,e_{n}\right)  $ is the canonical \emph{(}%
orthogonal\emph{) }basis of $\mathbb{R}^{n}$. Simultaneously, $\mathbb{R}^{n}$
is a vector lattice with respect to the coordinatewise ordering. That is,
\[
f\geq0\text{ in }\mathbb{R}^{n}\text{ if and only if }\left\langle
f,e_{k}\right\rangle \in\mathbb{R}^{+}\text{ for all }k\in\left\{
1,...,n\right\}  .
\]
Consequently, if $f,g\geq0$ in $\mathbb{R}^{n}$ then $\left\langle
f,g\right\rangle \in\mathbb{R}^{+}$. Furthermore, let $f,g\in\mathbb{R}^{n}$
such that $f\wedge g=0$. Whence,%
\[
\min\left\{  \left\langle f,e_{k}\right\rangle ,\left\langle g,e_{k}%
\right\rangle \right\}  =0\text{ for all }k\in\left\{  1,2,...,n\right\}  .
\]
Therefore, $\left\langle f,g\right\rangle =0$ meaning that the inner product
on $\mathbb{R}^{n}$ is an $\mathbb{R}$-valued orthosymmetric product. Thus,
the Euclidean space $\mathbb{R}^{n}$ is an orthosymmetric space over
$\mathbb{R}$.
\end{example}

\begin{quote}
\textit{Beginning with the next lines, we shall impose the blanket assumption
that all orthosymmetric spaces under consideration are taken over the fixed
Archimedean vector lattice }$\mathbb{V}$ (\textit{unless otherwise stated
explicitly}).
\end{quote}

\noindent The following property is useful for later purposes.

\begin{lemma}
\label{vp}Let $L$ be an orthosymmetric space. Then,%
\[
\left\langle f,f\right\rangle =\left\langle \left\vert f\right\vert
,\left\vert f\right\vert \right\rangle \in\mathbb{V}^{+}\text{ for all }f\in
L.
\]

\end{lemma}

\begin{proof}
If $f\in L$ then $f^{+}\wedge f^{-}=0$. It follows that%
\[
\left\langle f^{+},f^{-}\right\rangle =\left\langle f^{-},f^{+}\right\rangle
=0.
\]
Hence,%
\begin{align*}
\left\langle f,f\right\rangle  &  =\left\langle f^{+}-f^{-},f^{+}%
-f^{-}\right\rangle =\left\langle f^{+},f^{+}\right\rangle +\left\langle
f^{-},f^{-}\right\rangle \\
&  =\left\langle f^{+}+f^{-},f^{+}+f^{-}\right\rangle =\left\langle \left\vert
f\right\vert ,\left\vert f\right\vert \right\rangle \in\mathbb{V}^{+}.
\end{align*}
This is the desired result.
\end{proof}

At first sight, it might seem that it is easy to establish the following
remarkable property of orthosymmetric spaces. However, all proofs that can be
found in the literature are quite involved and far from being trivial (see,
e.g., Corollary $2$ in \cite{BR00} and Theorem $3.8.14$ in \cite{S10}). By the
way, it should be pointed out that this property is based on the fact that
$\mathbb{V}$ is Archimedean.

\begin{lemma}
\label{Steinberg}Let $L$ be an orthosymmetric space. Then,%
\[
\left\langle f,g\right\rangle =\left\langle g,f\right\rangle \text{ for all
}f,g\in L.
\]

\end{lemma}

Roughly speaking, any $\mathbb{V}$-valued orthosymmetric product on a vector
lattice is symmetric (a multidimensional version of Lemma \ref{Steinberg} can
be found in \cite{B02}). We emphasize that results in Lemmas \ref{vp} and
\ref{Steinberg} could be used below without further mention.

Before proceeding our investigation, we note that our next terminology comes
in part from the Inner Product Theory linguistic repertoire (see \cite{B74}).
Let $L$ be an orthosymmetric space. An element $f$ in $L$ is said to
be\textsl{ neutral} if $\left\langle f,f\right\rangle =0$. Obviously, the zero
vector is neutral. The set of all neutral elements in $L$ is called the
\textsl{neutral part} of $L$ and is denoted by $L^{0}$. Namely,%
\[
L^{0}=\left\{  f\in L:\left\langle f,f\right\rangle =0\right\}  .
\]
The neutral part of $L$ has a nice characterization.

\begin{lemma}
\label{key}Let $L$ be an orthosymmetric space. Then,%
\[
L^{0}=\left\{  f\in L:\left\langle f,g\right\rangle =0\text{ for all }g\in
L\right\}  .
\]

\end{lemma}

\begin{proof}
Obviously, if $f\in L$ with $\left\langle f,g\right\rangle =0$ for all $g\in
L$, then $\left\langle f,f\right\rangle =0$ so $f\in L^{0}$. Conversely, let
$f\in L^{0}$ and pick $g\in L$. Choose $n\in\mathbb{N}$ and observe that%
\[
0\leq\left\langle g-nf,g-nf\right\rangle =\left\langle g,g\right\rangle
-2n\left\langle f,g\right\rangle +n^{2}\left\langle f,f\right\rangle
=\left\langle g,g\right\rangle -2n\left\langle f,g\right\rangle .
\]
Therefore,%
\[
2n\left\langle f,g\right\rangle \leq\left\langle g,g\right\rangle \text{ for
all }n\in\mathbb{N}.
\]
Replacing $f$ by $-f$ in the above inequality, we obtain%
\[
-2n\left\langle f,g\right\rangle \leq\left\langle g,g\right\rangle \text{ for
all }n\in\mathbb{N}.
\]
But then $\left\langle f,g\right\rangle =0$ because $\mathbb{V}$ is
Archimedean. The proof is complete.
\end{proof}

An interesting lattice-ordered property of the neutral part of an
orthosymmetric space is obtained as a consequence of previous lemma.

\begin{theorem}
\label{neutral}The neutral part $L^{0}$ of an orthosymmetric space $L$ is an
ideal in $L$.
\end{theorem}

\begin{proof}
Let $r$ be a real number and $f,g\in L^{0}$. Then,%
\[
\left\langle f+rg,f+rg\right\rangle =\left\langle f,f\right\rangle
+2r\left\langle f,g\right\rangle +r^{2}\left\langle g,g\right\rangle =0
\]
(where we use Lemma \ref{key}). It follows that $f+rg\in L^{0}$ and so $L^{0}$
is a vector subspace of $L$.

Secondly, let $f\in L^{0}$ and observe that%
\[
\left\langle \left\vert f\right\vert ,\left\vert f\right\vert \right\rangle
=\left\langle f,f\right\rangle =0.
\]
Hence, $\left\vert f\right\vert \in L^{0}$ and thus $L^{0}$ is a vector
sublattice of $L$.

Finally, let $f,g\in L$ such that $0\leq f\leq g$ and $g\in L^{0}$. Whence,%
\[
0\leq\left\langle f,f\right\rangle \leq\left\langle f,g\right\rangle
\leq\left\langle g,g\right\rangle =0.
\]
This yields that $L^{0}$ is a solid in $L$ and finishes the proof.
\end{proof}

An orthosymmetric space need not be Archimedean as it is shown in the next example.

\begin{example}
\label{lexi}Assume that the Euclidean plan $\mathbb{R}^{2}$ is endowed with
its lexicographic ordering. Hence, $\mathbb{R}^{2}$ is a non-Archimedean
vector lattice. Put%
\[
\left\langle f,g\right\rangle =x_{1}y_{1}\text{ for all }f=\left(  x_{1}%
,x_{2}\right)  ,g=\left(  y_{1},y_{2}\right)  \text{ in }\mathbb{R}^{2}.
\]
Since $\mathbb{R}^{2}$ is totally ordered \emph{(}i.e., linearly
ordered\emph{)}, this formula defines an $\mathbb{R}$-valued orthosymmetric
product on $\mathbb{R}^{2}$. This means that $\mathbb{R}^{2}$ is a
non-Archimedean orthosymmetric space over $\mathbb{R}$.
\end{example}

The orthosymmetric space $L$ is said to be \textsl{definite}\textit{ }if its
neutral part is trivial. That is to say, $L$ is definite if and only if%
\[
\left\langle f,f\right\rangle =0\text{ in }\mathbb{V}\text{ implies }f=0\text{
in }L.
\]
Definite orthosymmetric spaces have a better behavior as explained in the following.

\begin{proposition}
\label{arch}Any definite orthosymmetric space is Archimedean.
\end{proposition}

\begin{proof}
Let $L$ be a definite orthosymmetric space and choose $f,g\in L^{+}$ with%
\[
0\leq nf\leq g\text{ for all }n\in\mathbb{N}.
\]
Pick $n\in\mathbb{N}$ and observe that $g-nf\in L^{+}$. So,%
\[
0\leq\left\langle g-nf,f\right\rangle =\left\langle g,f\right\rangle
-n\left\langle f,f\right\rangle .
\]
Hence,%
\[
0\leq n\left\langle f,f\right\rangle \leq\left\langle g,f\right\rangle \text{
for all }n\in\mathbb{N}.
\]
Since $\mathbb{V}$ is Archimedean, we get $\left\langle f,f\right\rangle =0$.
But then $f=0$ because $L$ is definite.
\end{proof}

By Theorem \ref{neutral}, the neutral part $L^{0}$ is an ideal in $L$. Hence,
we may consider the quotient vector lattice $L/L^{0}$ (see Chapter $9$ in
\cite{LZ71}). The equivalence class (i.e., the residue class)%
\[
f+L^{0}=\left\{  f+g:g\in L^{0}\right\}
\]
of a vector $f\in L$ is denoted by $\left[  f\right]  $. It turns out that
$L/L^{0}$ is automatically equipped with a structure of an orthosymmetric space.

\begin{theorem}
\label{quo}Let $L$ be an orthosymmetric space. Then $L/L^{0}$ is a definite
orthosymmetric space with respect to the $\mathbb{V}$-valued orthosymmetric
product given by%
\[
\left\langle \left[  f\right]  ,\left[  g\right]  \right\rangle =\left\langle
f,g\right\rangle \text{ for all }f,g\in L.
\]

\end{theorem}

\begin{proof}
First of all, let's prove that the function that takes each ordered pair
$\left(  \left[  f\right]  ,\left[  g\right]  \right)  $ of elements of
$L/L^{0}$ to a vector $\left\langle \left[  f\right]  ,\left[  g\right]
\right\rangle $ of $\mathbb{V}$ and given%
\begin{equation}
\left\langle \left[  f\right]  ,\left[  g\right]  \right\rangle =\left\langle
f,g\right\rangle \text{ for all }f,g\in L \tag{$\ast$}%
\end{equation}
is well-defined. Indeed, choose $f,g\in L$ and $h,k\in L^{0}$. By Lemma
\ref{key}, we have%
\[
\left\langle f,k\right\rangle =\left\langle h,g\right\rangle =\left\langle
h,k\right\rangle =0.
\]
So,%
\[
\left\langle f+h,g+k\right\rangle =\left\langle f,g\right\rangle +\left\langle
f,k\right\rangle +\left\langle h,g\right\rangle +\left\langle h,k\right\rangle
=\left\langle f,g\right\rangle .
\]
We derive that the function given by $\left(  \ast\right)  $ is well-defined
(its bilinearity is obvious). Now, let $f,g\in L$ such that $\left[  f\right]
,\left[  g\right]  $ are positive in $L/L^{0}$. Hence, there exist $h,k\in
L^{0}$ such that $h\leq f$ and $k\leq g$. Whence, $0\leq f-h$ and $0\leq g-k$
from which it follows that%
\[
0\leq\left\langle f-h,g-k\right\rangle =\left\langle f,g\right\rangle
-\left\langle f,k\right\rangle -\left\langle h,g\right\rangle +\left\langle
h,k\right\rangle =\left\langle f,g\right\rangle =\left\langle \left[
f\right]  ,\left[  g\right]  \right\rangle .
\]
Moreover, if $\left[  f\right]  \wedge\left[  g\right]  =\left[  0\right]  $
in $L/L^{0}$ then
\[
\left[  f\wedge g\right]  =\left[  f\right]  \wedge\left[  g\right]  =\left[
0\right]  .
\]
This means that $f\wedge g\in L^{0}$. This together with Lemma \ref{key}
yields quickly that
\[
\left\langle \left[  f\right]  ,\left[  g\right]  \right\rangle =\left\langle
f,g\right\rangle =\left\langle f-f\wedge g,g-f\wedge g\right\rangle .
\]
But then $\left\langle \left[  f\right]  ,\left[  g\right]  \right\rangle =0$
because%
\[
\left(  f-f\wedge g\right)  \wedge\left(  g-f\wedge g\right)  =0.
\]
Accordingly, $L/L^{0}$ is an orthosymmetric space. It remains to show that
$L/L^{0}$ is definite. To see this, let $f\in L$ such that $\left\langle
\left[  f\right]  ,\left[  f\right]  \right\rangle =0$. Hence, $\left\langle
f,f\right\rangle =0$ from which it follows that $f\in L^{0}$, so $\left[
f\right]  =\left[  0\right]  $. This completes the proof.
\end{proof}

Taking into account Proposition \ref{arch} and Theorem \ref{quo}, we infer
directly that the quotient vector lattice $L/L^{0}$ is Archimedean and so
$L^{0}$ is a uniformly closed ideal in the vector lattice $L$ (see Theorem
$60.2$ in \cite{LZ71}).

\section{Multiplication operators in orthosymmetric product spaces}

Let $M$ be an ordered vector space. A vector subspace $V$ of $M$ is called a
\textit{lattice-subspace} of $M$ if $V$ is a vector lattice with respect to
the ordering inherited from $M$ (see Definition $5.58$ in \cite{AA02}). On the
other hand, in general, we cannot talk about $V$ being a vector sublattice of
$M$ as the latter space need not be a vector lattice. In order to cope with
this terminology problem, we call after Abramovich and Wickstead in
\cite{AW94} the lattice-subspace $V$ of $M$ a \textsl{generalized vector
sublattice} of $M$ if the supremum in $M$ of each $v,w\in V$ exists and
coincides with their supremum in $V$. Hence, if $M$ turns out to be a vector
lattice, then the word `generalized' becomes superfluous. Moreover, it is
trivial that every generalized vector sublattice of $M$ is a lattice-subspace.
Nevertheless, the converse is not true as we can see in the example provided
in \cite[Page $229$]{AA02}.

We start this section with the following general lemma which is presumably
well-known, though we have not been able to locate a precise reference for it.
As usual, the kernel and the range of any linear operator $T$ are denoted by
$\ker T$ and $\operatorname{Im}T$, respectively.

\begin{lemma}
\label{gene}Let $L$ be a vector lattice and $M$ be an ordered vector space.
Suppose that $T:L\rightarrow M$ is a positive operator such that

\begin{enumerate}
\item[\emph{(i)}] $\ker T$ is an vector sublattice of $L$, and

\item[\emph{(ii)}] $f^{-}\in\ker T$ for all $f\in L$ with $Tf\in M^{+}$.
\end{enumerate}

\noindent Then $\operatorname{Im}T$ is a lattice-subspace of $M$ and $T$ is a
lattice homomorphism from $L$ onto $\operatorname{Im}T$.
\end{lemma}

\begin{proof}
Let $f,g\in L$. Since $T$ is positive,%
\[
T\left(  f\vee g\right)  \geq Tf\text{ and }T\left(  f\vee g\right)  \geq Tg.
\]
This means that $T\left(  f\vee g\right)  $ is an upper bound of $\left\{
Tf,Tg\right\}  $ in $\operatorname{Im}T$. Let $v\in\operatorname{Im}T$ another
upper bound of $\left\{  Tf,Tg\right\}  $ in $\operatorname{Im}T$. There
exists $h\in L$ such that $v=Th$. Since%
\[
Th=v\geq Tf\text{ and }Th=v\geq Tg,
\]
we get%
\[
T\left(  h-f\right)  \geq0\text{ and }T\left(  h-g\right)  \geq0.
\]
Therefore,%
\[
\left(  h-f\right)  ^{-}\in\ker T\text{ and }\left(  h-g\right)  ^{-}\in\ker
T.
\]
But then%
\[
\left(  h-\left(  f\vee g\right)  \right)  ^{-}=\left(  h-f\right)  ^{-}%
\vee\left(  h-g\right)  ^{-}\in\ker T
\]
because $\ker T$ is a vector sublattice of $M$. It follows that%
\[
T\left(  h-\left(  \left(  f\vee g\right)  \right)  \right)  =T\left(
f-\left(  \left(  f\vee g\right)  \right)  ^{+}\right)  \geq0
\]
and so%
\[
v=Th\geq T\left(  f\vee g\right)  .
\]
We derive that $T\left(  f\vee g\right)  $ is a supremum of $\left\{
Tf,Tg\right\}  $ in $\operatorname{Im}T$. That is,%
\[
T\left(  f\vee g\right)  =Tf\vee Tg\text{ in }\operatorname{Im}T
\]
and the proof is complete.
\end{proof}

Lemma \ref{gene} does not hold without the condition $\mathrm{(ii)}$. An
example in this direction is given next.

\begin{example}
\label{integ}Let $L=M=C\left[  0,\pi\right]  $, where $C\left[  0,\pi\right]
$ is the Archimedean vector lattice of all real-valued continuous functions on
the real interval $\left[  0,\pi\right]  $. Define $T:L\rightarrow M$ by
putting%
\[
\left(  Tf\right)  \left(  x\right)  =\int_{0}^{x}f\left(  y\right)
\mathrm{d}y\text{ for all }f\in L\text{ and }x\in\left[  0,\pi\right]  .
\]
Obviously, $T$ is a positive operator. Moreover, $T$ is one-to-one and so
$\ker T=\left\{  0\right\}  $ is an ideal in $L$. However, if $f\in L$ is
defined by%
\[
f\left(  x\right)  =2\sin x\cos x\text{ for all }x\in\left[  0,\pi\right]  ,
\]
then%
\[
\left(  Tf\right)  \left(  x\right)  =\left(  \sin x\right)  ^{2}\text{ for
all }x\in\left[  0,\pi\right]  .
\]
Furthermore,%
\[
f^{-}\left(  x\right)  =\left\{
\begin{array}
[c]{l}%
2\sin x\cos x\text{ if }x\in\left[  0,\pi/2\right] \\
\\
0\text{ if }x\in\left[  \pi/2,\pi\right]  .
\end{array}
\right.
\]
Thus,%
\[
\left(  T\left(  f^{-}\right)  \right)  \left(  \pi/2\right)  =\int_{0}%
^{\pi/2}2\sin x\cos x\mathrm{d}x=1\neq0.
\]
Hence, $f^{-}\notin\ker T$. Observe now that%
\[
\operatorname{Im}T=\left\{  f\in L:f\left(  0\right)  =0\text{ and }f\text{
continuously differentiable}\right\}
\]
is not a lattice-subspace of $M$.
\end{example}

Hence, Example \ref{integ} proves that the condition $\mathrm{(ii)}$ is not
redundant in Lemma \ref{gene}. In spite of that, the following observation
deserves particular attention.

\begin{remark}
Let $L$ be a vector lattice, $M$ be an ordered vector space, and
$T:L\rightarrow M$ be a positive operator such that $\ker T$ is a vector
sublattice of $L$. We derive quickly that $\ker T$ is an ideal in $L$. Hence,
we may speak about the vector lattice $L/\ker T$ \emph{(}see
\emph{\cite[Chapter $9$]{LZ71})}. In this situation, it is not hard to see
that an ordering can be defined on $\operatorname{Im}T$ by putting%
\[
Tf\preccurlyeq Tg\text{ in }\operatorname{Im}T\text{ if and only if }g-f\geq
h\text{ for some }h\in\ker T.
\]
This ordering makes $\operatorname{Im}T$ into a vector lattice which is
lattice isomorphic with $L/\ker T$. The lattice operations in
$\operatorname{Im}T$ are given pointwise as follows%
\[
Tf\curlyvee Tg=T\left(  f\vee g\right)  \text{ and }Tf\curlywedge Tg=T\left(
f\wedge g\right)  \text{ for all }f,g\in L.
\]
However, $\operatorname{Im}T$ need not be, in general, a lattice-subspace of
$M$. As a matter of fact, \emph{Example \ref{integ}} illustrates this
situation. Indeed, we have observed already that%
\[
\operatorname{Im}T=\left\{  f\in C\left[  0,\pi\right]  :f\left(  0\right)
=0\text{ and }f\text{ continuously differentiable}\right\}
\]
is not a lattice-subspace of $C\left[  0,\pi\right]  $. Nevertheless,
$\operatorname{Im}T$ is a vector lattice with respect to the `new' ordering
defined by%
\[
f\preccurlyeq g\text{ if and only if }f^{\prime}\leq g^{\prime},
\]
where $f^{\prime}$ and $g^{\prime}$ denote the derivative of $f$ and $g$,
respectively. Moreover, the lattice operations in this vector lattice are
given by%
\[
\left(  f\curlyvee g\right)  \left(  x\right)  =\int_{0}^{x}\left(  f^{\prime
}\vee g^{\prime}\right)  \left(  y\right)  \mathrm{d}y\text{ and }\left(
f\curlywedge g\right)  \left(  x\right)  =\int_{0}^{x}\left(  f^{\prime}\wedge
g^{\prime}\right)  \left(  y\right)  \mathrm{d}y
\]
for all $f,g\in\operatorname{Im}T$ and $x\in\left[  0,\pi\right]  $.
\end{remark}

Henceforth, $L$ stands for an orthosymmetric space over the Archimedean vector
lattice $\mathbb{V}$ and $\mathcal{L}^{+}\left(  L,\mathbb{V}\right)  $
denotes the set of all positive operators from $L$ into $\mathbb{V}$. It could
be helpful to recall that an operator $T:L\rightarrow\mathbb{V}$ is said to be
\textsl{regular} if $T=R-S$ for some $R,S\in\mathcal{L}^{+}\left(
L,\mathbb{V}\right)  $. The set $\mathcal{L}_{r}\left(  L,\mathbb{V}\right)  $
of all regular operators from $L$ into $\mathbb{V}$ is an Archimedean ordered
vector space with respect to the pointwise addition, scalar multiplication,
and ordering. By the way, $\mathcal{L}^{+}\left(  L,,\mathbb{V}\right)  $ is
the positive cone of $\mathcal{L}_{r}\left(  L,\mathbb{V}\right)  $.

At this point, let $f\in L$ and define a map $\Phi_{f}:L\rightarrow\mathbb{V}$
by putting%
\[
\Phi_{f}g=\left\langle f,g\right\rangle \text{ for all }g\in L.
\]
Obviously, $\Phi_{f}$ is a linear operator, which is referred to as a
\textsl{multiplication operator} on the orthosymmetric space $L$. Further
elementary (but very useful) properties of such operators are gathered next.

\begin{lemma}
\label{tech}Let $L$ be an orthosymmetric space and $f\in L$. Then the
following hold.

\begin{enumerate}
\item[\emph{(i)}] $\Phi_{f}\in\mathcal{L}^{+}\left(  L,\mathbb{V}\right)  $
whenever $f\in L^{+}$.

\item[\emph{(ii)}] $\Phi_{f}=\Phi_{f^{+}}-\Phi_{f^{-}}\in\mathcal{L}%
_{r}\left(  L,\mathbb{V}\right)  $.

\item[\emph{(iii)}] $\Phi_{f}=0$ if and only if $f\in L^{0}$.

\item[\emph{(iv)}] $\Phi_{f}\in\mathcal{L}^{+}\left(  L,\mathbb{V}\right)  $
if and only if $f^{-}\in L^{0}$.
\end{enumerate}
\end{lemma}

\begin{proof}
$\mathrm{(i)}$ This follows immediately from the positivity of the
$\mathbb{V}$-valued orthosymmetric product on $L$.

$\mathrm{(ii)}$ If $f,g\in L$ then%
\begin{align*}
\Phi_{f}g  &  =\left\langle f,g\right\rangle =\left\langle f^{+}%
-f^{-},g\right\rangle =\left\langle f^{+},g\right\rangle -\left\langle
f^{-},g\right\rangle \\
&  =\Phi_{f^{+}}g-\Phi_{f^{-}}g=\left(  \Phi_{f^{+}}-\Phi_{f^{-}}\right)  g.
\end{align*}
Thus, $\Phi_{f}=\Phi_{f^{+}}-\Phi_{f^{-}}$. Moreover, $\Phi_{f^{+}}%
,\Phi_{f^{-}}\in\mathcal{L}^{+}\left(  L,\mathbb{V}\right)  $ as $f^{+}%
,f^{-}\in L^{+}$ (where we use $\mathrm{(i)}$). It follows that $\Phi_{f}$ is regular.

$\mathrm{(iii)}$ This is a direct consequence of Lemma \ref{key}.

$\mathrm{(iv)}$ If $f^{-}\in L^{0}$ then, by $\mathrm{(iii)}$, $\Phi_{f^{-}%
}=0$. Using $\mathrm{(ii)}$ then $\mathrm{(i)}$, we get $\Phi_{f}=\Phi_{f^{+}%
}\in\mathcal{L}^{+}\left(  L,\mathbb{V}\right)  $. Conversely, suppose that
$\Phi_{f}\in\mathcal{L}^{+}\left(  L,\mathbb{V}\right)  $. Then,%
\begin{align*}
0  &  \leq\Phi_{f}\left(  f^{-}\right)  =\Phi_{f^{+}}\left(  f^{-}\right)
-\Phi_{f^{-}}\left(  f^{-}\right) \\
&  =\left\langle f^{+},f^{-}\right\rangle -\left\langle f^{-},f^{-}%
\right\rangle =-\left\langle f^{-},f^{-}\right\rangle \leq0.
\end{align*}
We derive that $\left\langle f^{-},f^{-}\right\rangle =0$ so $f^{-}\in L^{0}$,
as required.
\end{proof}

As we shall see in what follows, it turns out that the set%
\[
\mathcal{M}\left(  L,\mathbb{V}\right)  =\left\{  \Phi_{f}:f\in L\right\}
\]
of all multiplication operators on the orthosymmetric space $L$ enjoys a very
interesting lattice-ordered structure.

\begin{theorem}
\label{main}Let $L$ be an orthosymmetric space. Then $\mathcal{M}\left(
L,\mathbb{V}\right)  $ is a generalized vector sublattice of $\mathcal{L}%
_{r}\left(  L,\mathbb{V}\right)  $ and the map $\Phi:L\rightarrow
\mathcal{M}\left(  L,\mathbb{V}\right)  $ defined by%
\[
\Phi f=\Phi_{f}\text{ for all }f\in L
\]
is a surjective lattice homomorphism with $\ker\Phi=L^{0}$. In particular, the
vector lattice $L/L^{0}$ and $\mathcal{M}\left(  L,\mathbb{V}\right)  $ are
lattice isomorphic.
\end{theorem}

\begin{proof}
We have seen in Lemma \ref{tech} $\mathrm{(ii)}$ that $\mathcal{M}\left(
L,\mathbb{V}\right)  $ is contained in $\mathcal{L}_{r}\left(  L,\mathbb{V}%
\right)  $. Moreover, it is a simple exercise to check that $\mathcal{M}%
\left(  L,\mathbb{V}\right)  $ is a vector subspace of $\mathcal{L}_{r}\left(
L,\mathbb{V}\right)  $. Also, we may check directly that $\Phi$ is a linear
operator and, by Lemma \ref{tech} $\mathrm{(i)}$, $\Phi$ is positive.
Furthermore, Lemma \ref{tech} $\mathrm{(iii)}$ yields that $\ker\Phi=L^{0}$.
In particular, $\ker\Phi$ is a vector sublattice of $L$ (where we use Theorem
\ref{neutral}). Moreover, Lemma \ref{tech} $\mathrm{(iv)}$ shows that
$f^{-}\in\ker\Phi$ whenever $f\in L$ and $0\leq\Phi f\in\mathcal{M}\left(
L,\mathbb{V}\right)  $. Consequently, since $\operatorname{Im}\Phi
=\mathcal{M}\left(  L,\mathbb{V}\right)  $, all conditions of Lemma \ref{gene}
are fulfilled. So, $\mathcal{M}\left(  L,\mathbb{V}\right)  $ is a
lattice-subspace of $\mathcal{L}_{r}\left(  L,\mathbb{V}\right)  $ and $\Phi$
is a lattice homomorphism from $L$ onto $\mathcal{M}\left(  L,\mathbb{V}%
\right)  $. In particular, if $f\in L$ then the absolute value of $\Phi_{f}$
in $\mathcal{M}\left(  L,\mathbb{V}\right)  $ equals $\Phi_{\left\vert
f\right\vert }$. At this point, we claim that $\mathcal{M}\left(
L,\mathbb{V}\right)  $ is a generalized vector sublattice of $\mathcal{L}%
_{r}\left(  L,\mathbb{V}\right)  $. To this end, we shall prove that for each
$f\in L$ and $g\in L^{+}$ the set%
\[
A\left(  f,g\right)  =\left\{  \left\vert \Phi_{f}h\right\vert :h\in L\text{
and }\left\vert h\right\vert \leq g\right\}
\]
has $\Phi_{\left\vert f\right\vert }g$ as a supremum in $\mathbb{V}$. The
Riesz-Kantorovich formula (see, e.g. Theorem $1.14$ in \cite{AB06}) will thus
give the desired result.

So, let $f\in L$ and $g\in L^{+}$. If $h\in L$ with $\left\vert h\right\vert
\leq g$, then%
\[
\left\vert \Phi_{f}h\right\vert =\left\vert \left\langle f,h\right\rangle
\right\vert \leq\left\langle \left\vert f\right\vert ,\left\vert h\right\vert
\right\rangle \leq\left\langle \left\vert f\right\vert ,g\right\rangle
=\Phi_{\left\vert f\right\vert }g.
\]
That is, $\Phi_{\left\vert f\right\vert }g$ is an upper bound of $A\left(
f,g\right)  $ in $\mathbb{V}$. We now proceed into three steps.

\underline{\textsl{Step 1}}\quad Assume that $L$ is Dedekind-complete. In
particular, any band is a projection band. Let $P_{\left\vert f\right\vert }$
denote the order projection on the principal band in $L$ generated by
$\left\vert f\right\vert $. In particular, we have%
\[
\left\vert P_{\left\vert f\right\vert }g\right\vert \leq g.
\]
We derive, via an elementary calculation, that
\[
\Phi_{\left\vert f\right\vert }g=\left\langle \left\vert f\right\vert
,g\right\rangle =\left\langle f,P_{\left\vert f\right\vert }g\right\rangle
=\left\vert \left\langle f,P_{\left\vert f\right\vert }g\right\rangle
\right\vert =\left\vert \Phi_{f}\left(  P_{\left\vert f\right\vert }g\right)
\right\vert \in A\left(  f,g\right)  .
\]
It follows that the set $A\left(  f,g\right)  $ has a supremum in $\mathbb{V}$
and%
\[
\Phi_{\left\vert f\right\vert }g=\sup A\left(  f,g\right)  =\sup\left\{
\left\vert \Phi_{f}h\right\vert :h\in L\text{ and }\left\vert h\right\vert
\leq g\right\}  .
\]
This completes the first step.

\underline{\textsl{Step 2}}\quad Suppose that $L$ is Archimedean. Let
$L^{\delta}$ and $\mathbb{V}^{\delta}$ denote the Dedekind-completions of $L$
and $\mathbb{V}$, respectively. There exists a $\mathbb{V}^{\delta}$-valued
orthosymmetric product on $L^{\delta}$ which extends the $L$'s (see Theorem
$4.1$ in \cite{BK07}). The image under this product of each ordered pair
$\left(  f,g\right)  $ of elements in $L^{\delta}$ is denoted by $\left\langle
f,g\right\rangle ^{\delta}$. Furthermore, we define $\Phi_{f}^{\delta
}:L^{\delta}\rightarrow\mathbb{V}^{\delta}$ by putting%
\[
\Phi_{f}^{\delta}g=\left\langle f,g\right\rangle ^{\delta}\text{ for all }g\in
L^{\delta}.
\]
Let $T\in\mathcal{L}_{r}\left(  L,\mathbb{V}\right)  $ such that%
\[
T\geq\pm\Phi_{f}\text{ in }\mathcal{L}_{r}\left(  L,\mathbb{V}\right)  .
\]
In particular, $T$ is positive. Choose a positive extension $T^{\delta}%
\in\mathcal{L}_{r}\left(  L^{\delta},\mathbb{V}^{\delta}\right)  $ of $T$
(see, e.g., Corollary $1.5.9$ in \cite{M91}). By the first case (as
$L^{\delta}$ is Dedekind-complete), we get%
\[
T^{\delta}\geq\Phi_{\left\vert f\right\vert }\text{ in }\mathcal{L}_{r}\left(
L^{\delta},\mathbb{V}^{\delta}\right)  .
\]
Thus, if $g\in L^{+}$ then%
\[
Tg=T^{\delta}g\geq\Phi_{\left\vert f\right\vert }^{\delta}g=\left\langle
\left\vert f\right\vert ,g\right\rangle ^{\delta}=\left\langle \left\vert
f\right\vert ,g\right\rangle =\Phi_{\left\vert f\right\vert }g.
\]
This yields that%
\[
T\geq\Phi_{\left\vert f\right\vert }\text{ in }\mathcal{L}_{r}\left(
L,\mathbb{V}\right)
\]
and so%
\[
\left\vert \Phi_{f}\right\vert =\Phi_{\left\vert f\right\vert }\text{ in
}\mathcal{L}_{r}\left(  L,\mathbb{V}\right)  .
\]
We focus next on the general case.

\underline{\textsl{Step 3}}\quad Here we do not assume any extra condition on
$L$. Set%
\[
A\left(  \left[  f\right]  ,\left[  g\right]  \right)  =\left\{  \left\vert
\Phi_{f}h\right\vert :h\in L\text{ and }\left\vert \left[  h\right]
\right\vert \leq\left[  g\right]  \text{ in }L/L^{0}\right\}  .
\]
We claim that%
\[
A\left(  f,g\right)  =A\left(  \left[  f\right]  ,\left[  g\right]  \right)
.
\]
The inclusion%
\[
A\left(  f,g\right)  \subset A\left(  \left[  f\right]  ,\left[  g\right]
\right)
\]
being obvious, we prove the converse inclusion. Hence, let $h\in L$ such that
$\left\vert \left[  h\right]  \right\vert \leq\left[  g\right]  $ in $L/L^{0}%
$. From Theorem $59.1$ in \cite{LZ71}, it follows that there $k\in L^{0}$ such
that $\left\vert k\right\vert \leq\left\vert h\right\vert $ and $\left\vert
h-k\right\vert \leq\left\vert g\right\vert $. But then%
\[
\left\vert \Phi_{f}h\right\vert =\left\vert \Phi_{f}\left(  h-k\right)
\right\vert \in A\left(  f,g\right)
\]
and the required inclusion follows. Theorem \ref{quo} together with the
Archimedean case leads to%
\begin{align*}
\Phi_{\left\vert f\right\vert }g  &  =\left\langle \left\vert f\right\vert
,g\right\rangle =\left\langle \left[  \left\vert f\right\vert \right]
,\left[  g\right]  \right\rangle =\left\langle \left\vert \left[  f\right]
\right\vert ,\left[  g\right]  \right\rangle \\
&  =\Phi_{\left\vert \left[  f\right]  \right\vert }\left[  g\right]  =\sup
A\left(  \left[  f\right]  ,\left[  g\right]  \right)  =\sup A\left(
f,g\right)  .
\end{align*}
This completes the proof of theorem.
\end{proof}

In what follows we shall discuss an example which illustrates Theorem
\ref{main}.

\begin{example}
\label{kaplan}Let $C\left(  \mathbb{N}\right)  $ be the Archimedean vector
lattice of all sequences of real numbers. The vector sublattice of $C\left(
\mathbb{N}\right)  $ of all bounded sequences is denoted by $C^{\ast}\left(
\mathbb{N}\right)  $ \emph{(}here, we follow notations from \emph{\cite{GJ76}%
)}. Define $T\in\mathcal{L}_{r}\left(  C^{\ast}\left(  \mathbb{N}\right)
,C\left(  \mathbb{N}\right)  \right)  $ by%
\[
\left(  Tf\right)  \left(  n\right)  =\left\{
\begin{array}
[c]{l}%
f\left(  n\right)  -f\left(  n+1\right)  \text{ if }n\in\mathbb{N}\text{ is
odd}\\
\\
0\text{ if }n\in\mathbb{N}\text{ is even}%
\end{array}
\right.  \text{ for all }f\in C^{\ast}\left(  \mathbb{N}\right)  .
\]
Since $C\left(  \mathbb{N}\right)  $ is Dedekind-complete, $\mathcal{L}%
_{r}\left(  C^{\ast}\left(  \mathbb{N}\right)  ,C\left(  \mathbb{N}\right)
\right)  $ is a vector lattice and thus $T$ has an absolute value $\left\vert
T\right\vert $. We intend to calculate $\left\vert T\right\vert $. Consider
$u,v\in C^{\ast}\left(  \mathbb{N}\right)  $ with%
\[
u=\left\{
\begin{array}
[c]{l}%
1\text{ if }n\text{ is odd}\\
\\
0\text{ if }n\text{ is even}%
\end{array}
\right.  \text{ and }v=\left\{
\begin{array}
[c]{l}%
0\text{ if }n\text{ is odd}\\
\\
1\text{ if }n\text{ is even.}%
\end{array}
\right.
\]
Moreover, define the shift operator $S\in\mathcal{L}^{+}\left(  C^{\ast
}\left(  \mathbb{N}\right)  ,C\left(  \mathbb{N}\right)  \right)  $ by putting%
\[
\left(  Sf\right)  \left(  n\right)  =f\left(  n+1\right)  \text{ for all
}f\in C^{\ast}\left(  \mathbb{N}\right)  \text{ and }n\in\mathbb{N}.\text{ }%
\]
Also, if $f,g\in C^{\ast}\left(  \mathbb{N}\right)  $ then $fg$ is defined by%
\[
\left(  fg\right)  \left(  n\right)  =f\left(  n\right)  g\left(  n\right)
\text{ for all }f\in C^{\ast}\left(  \mathbb{N}\right)  \text{ and }%
n\in\mathbb{N}.
\]
Then, it is readily checked that the formula%
\[
\left\langle f,g\right\rangle =ufg+S\left(  vfg\right)  \text{ for all }f,g\in
L.
\]
makes $C^{\ast}\left(  \mathbb{N}\right)  $ into an orthosymmetric space over
$C\left(  \mathbb{N}\right)  $. An easy calculation yields that%
\[
Tf=\left\langle u-v,f\right\rangle =\Phi_{u-v}f\text{ for all }f\in L.
\]
By \emph{Theorem \ref{main}}, we derive that $T$ has an absolute value in
$\mathcal{L}_{r}\left(  C^{\ast}\left(  \mathbb{N}\right)  ,C\left(
\mathbb{N}\right)  \right)  $ which is given by%
\[
\left\vert T\right\vert =\left\vert \Phi_{u-v}\right\vert =\Phi_{\left\vert
u-v\right\vert }=\Phi_{u+v}.
\]
That is,%
\[
\left(  \left\vert T\right\vert f\right)  \left(  n\right)  =\left\{
\begin{array}
[c]{l}%
f\left(  n\right)  +f\left(  n+1\right)  \text{ if }n\in\mathbb{N}\text{ is
odd}\\
\\
0\text{ if }n\in\mathbb{N}\text{ is even}%
\end{array}
\right.  \text{ for all }f\in C^{\ast}\left(  \mathbb{N}\right)  .
\]
By the way, let $C\left(  \mathbb{N}^{\ast}\right)  $ be the vector sublattice
of $C\left(  \mathbb{N}\right)  $ of all convergent sequences \emph{(}where
$\mathbb{N}^{\ast}$ denote the point-one compactification of $\mathbb{N}%
$\emph{)}. Also, $T$ leaves $C\left(  \mathbb{N}^{\ast}\right)  $ invariant
and thus $T$ can be considered as an element of $\mathcal{L}_{r}\left(
C\left(  \mathbb{N}^{\ast}\right)  ,C\left(  \mathbb{N}^{\ast}\right)
\right)  $. Using an example by Kaplan in \emph{\cite{K73}} \emph{(}see also
\emph{Example} $1.17$ in \emph{\cite{AB06})}, it turns out that $T$ has no
absolute value in $\mathcal{L}_{r}\left(  C\left(  \mathbb{N}^{\ast}\right)
,C\left(  \mathbb{N}^{\ast}\right)  \right)  $.
\end{example}

The following consequence of Theorem \ref{main} is straightforward but worth
talking about. First recall that the orthosymmetric space is definite if
$L^{0}$ is trivial.

\begin{corollary}
Let $L$ be a definite orthosymmetric space. The map $T:L\rightarrow
\mathcal{M}\left(  L,\mathbb{V}\right)  $ defined by%
\[
\Phi f=\Phi_{f}\text{ for all }f\in L
\]
is a lattice isomorphism.
\end{corollary}

Roughly speaking, any definite orthosymmetric space $L$ can be embedded in
$\mathcal{L}_{r}\left(  L,\mathbb{V}\right)  $ as a generalized vector
sublattice. In particular, if $\mathbb{V}$ is in addition Dedekind-complete,
then $L$ has a vector sublattice copy in the Dedekind-complete $\mathcal{L}%
_{r}\left(  L,\mathbb{V}\right)  $. For instance, any definite orthosymmetric
space over $\mathbb{R}$ can be considered as a vector sublattice of its order dual.

\section{Adjoint operators on orthosymmetric spaces}

Also in this section, \textsl{all given orthosymmetric spaces are over the
Archimedean vector lattice} $\mathbb{V}$. Let $L,M$ be two orthosymmetric
spaces. The ordered vector space of all (linear) operators from $L$ into $M$
is denoted by $\mathcal{L}_{\mathrm{r}}\left(  L,M\right)  $ and by
$\mathcal{L}_{\mathrm{r}}\left(  L\right)  $ if $L=M$. Recall that we call a
\textsl{positive orthomorphism} on $L$ any positive operator $T\in
\mathcal{L}_{\mathrm{r}}\left(  L\right)  $ for which%
\[
f\wedge Tg=0\text{ for all }f,g\in L\text{ with }f\wedge g=0.
\]
An operator $T\in\mathcal{L}_{\mathrm{r}}\left(  L\right)  $ is called an
\textsl{orthomorphism}\textit{ }if $T=R-S$ for some positive orthomorphisms
$R,S$ on $L$. The set of all orthomorphisms on $L$ is denoted by
$\mathrm{Orth}\left(  L\right)  $. For orthomorphisms, the reader can consult
the Thesis \cite{P81} or Chapter $20$ in \cite{Z83} (further results can be
found in \cite{HP86,HP82,HP84,P84}).

It turns out that orthomorphisms on orthosymmetric spaces over the same
Archimedean vector lattice have an interesting property.

\begin{theorem}
\label{orth}Let $L$ be an orthosymmetric space. Then%
\[
\left\langle f,Tg\right\rangle =\left\langle Tf,g\right\rangle \text{ for all
}T\in\mathrm{Orth}\left(  L\right)  \text{ and }f,g\in L.
\]

\end{theorem}

\begin{proof}
Obviously, we can prove the formula only for positive orthomorphisms. Hence,
let $T$ be a positive orthomorphism on $L$ and define a bilinear map which
assigns the vector%
\[
\left\langle f,g\right\rangle _{T}=\left\langle f,Tg\right\rangle
\in\mathbb{V}%
\]
to each ordered pair $\left(  f,g\right)  \in L\times L$. Clearly, the above
formula defines a new $\mathbb{V}$-valued orthosymmetric product on $L$. In
view of Lemma \ref{Steinberg}, we conclude that if $f,g\in L$ then%
\[
\left\langle f,Tg\right\rangle =\left\langle f,g\right\rangle _{T}%
=\left\langle g,f\right\rangle _{T}=\left\langle g,Tf\right\rangle
=\left\langle Tf,g\right\rangle
\]
and the theorem follows.
\end{proof}

Theorem \ref{orth} motivates us to introduce the following concept. We say
that $T\in\mathcal{L}_{\mathrm{r}}\left(  L,M\right)  $ has an
\textsl{adjoint}\textit{ }in $\mathcal{L}_{\mathrm{r}}\left(  M,L\right)  $ if
there is $S\in\mathcal{L}_{\mathrm{r}}\left(  M,L\right)  $ for which%
\[
\left\langle Tf,g\right\rangle =\left\langle f,Sg\right\rangle \text{ for all
}f\in L\text{ and }g\in M.
\]
The subset of $\mathcal{L}_{\mathrm{r}}\left(  M,L\right)  $ of all adjoints
of $T\in\mathcal{L}_{\mathrm{r}}\left(  L,M\right)  $ is denoted by
$\mathrm{adj}\left(  T\right)  $. The operator $T\in\mathcal{L}_{\mathrm{r}%
}\left(  L\right)  $ is said to be \textsl{selfadjoint} if $T\in
\mathrm{adj}\left(  T\right)  $. It follows from Theorem \ref{orth} that any
orthomorphism on $L$ is selfadjoint. Of course, the converse is not valid,
i.e., a selfadjoint operator in $\mathcal{L}_{\mathrm{r}}\left(  L\right)  $
need not be an orthomorphism. Consider the orthosymmetric space $\mathbb{R}%
^{2}$ as defined in Example \ref{Euclidean}. The operator $T\in\mathcal{L}%
_{\mathrm{r}}\left(  \mathbb{R}^{2}\right)  $ given by the $2\times2$ matrix%
\[
T=\left(
\begin{array}
[c]{cc}%
1 & 2\\
2 & 0
\end{array}
\right)
\]
is selfadjoint but fails to be an orthomorphism. Recall by the way that any
orthomorphism on the vector lattice $\mathbb{R}^{n}$ ($n\in\mathbb{N}$) is
given by a diagonal matrix as shown in \cite[Exercice $141.7$]{Z83}. Next, we
shall show \textit{via }an example that $\mathrm{adj}\left(  T\right)  $ can
be empty.

\begin{example}
Put $L=\mathbb{V}=C\left[  -1,1\right]  $, where $C\left[  -1,1\right]  $ is
the Archimedean vector lattice of all real-valued continuous functions on the
real interval $\left[  0,1\right]  $. It is not hard to see that $L$ is an
orthosymmetric space over $\mathbb{V}$ under the $\mathbb{V}$-valued
orthosymmetric product given by%
\[
\left\langle f,g\right\rangle =fg\text{ for all }f,g\in L.
\]
Now, define $T\in\mathcal{L}_{\mathrm{r}}\left(  L\right)  $ by%
\[
\left(  Tf\right)  \left(  x\right)  =x\int_{-1}^{1}f\left(  s\right)
\mathrm{d}s\text{ for all }f\in L\text{ and }x\in\left[  -1,1\right]  .
\]
Suppose that $\mathrm{adj}\left(  T\right)  $ contains $R$. Let $f\in L$ and
put%
\[
g\left(  x\right)  =x\text{ for all }x\in\left[  -1,1\right]  .
\]
Hence, $Tg=0$ and so%
\[
x\left(  Rf\right)  \left(  x\right)  =\left\langle Rf,g\right\rangle
=\left\langle f,Tg\right\rangle =0\text{ for all }x\in\left[  -1,1\right]  .
\]
It follows that $R=0$ and thus%
\[
Tf=\left\langle Tf,\mathbf{1}\right\rangle =\left\langle f,R\mathbf{1}%
\right\rangle =0\text{ for all }f\in L
\]
\emph{(}here, $\mathbf{1}$ is the constant function whose is the constant
$1$\emph{)}. This is an obvious absurdity and thus $\mathrm{adj}\left(
T\right)  $ is empty.
\end{example}

Next, we provide an example in which we shall see that $\mathrm{adj}\left(
T\right)  $ may contain more than one element.

\begin{example}
Again, we put $L=\mathbb{V}=C\left[  -1,1\right]  $. For each $f,g\in L$ we
set%
\[
\left\langle f,g\right\rangle \left(  x\right)  =\left\{
\begin{array}
[c]{l}%
x%
{\displaystyle\int_{-1}^{0}}
\left(  fg\right)  \left(  s\right)  \mathrm{d}s\text{\quad if }x\in\left[
0,1\right] \\
\\
0\text{\quad if }x\in\left[  -1,0\right]  .
\end{array}
\right.
\]
It is an easy task to check that this formula makes $L$ into an orthosymmetric
space over $\mathbb{V}$. Define $T\in\mathcal{L}_{\mathrm{r}}\left(  L\right)
$ by%
\[
\left(  Tf\right)  \left(  x\right)  =x\int_{-1}^{0}f\left(  s\right)
\mathrm{d}s\text{ for all }f\in L\text{ and }x\in\left[  -1,1\right]  .
\]
Moreover, consider $R,S\in\mathcal{L}_{\mathrm{r}}\left(  L\right)  $ such
that%
\[
\left(  Rf\right)  \left(  x\right)  =\int_{-1}^{0}sf\left(  s\right)
\mathrm{d}s\text{ and }Sf=fw+Rf\text{ for all }f\in L\text{ and }x\in\left[
-1,1\right]  ,
\]
where%
\[
w\left(  x\right)  =0\text{ if }x\in\left[  -1,0\right]  \text{ and }w\left(
x\right)  =x\text{ if }x\in\left[  0,1\right]  .
\]
A direct calculation reveals that%
\[
\left\langle f,Tg\right\rangle =\left\langle Rf,g\right\rangle =\left\langle
Sf,g\right\rangle \text{ for all }f,g\in L.
\]
Hence, $R,S\in\mathrm{adj}\left(  T\right)  $ and $R\neq S$.
\end{example}

The fact that an operator $T\in\mathcal{L}_{\mathrm{r}}\left(  L,M\right)  $
could have more than one adjoint is rather inconvenient. In order to get
around the problem and thus make our theory more or less reasonable,
\textit{we shall assume from now on that the codomain orthosymmetric space
}$M$\textit{ is definite}. Recall here that the orthosymmetric space $M$ is
definite if $0$ is the only neutral element in $M$. This means that%
\[
M^{0}=\left\{  g\in M:\left\langle g,g\right\rangle =0\right\}  =\left\{
0\right\}  .
\]
As we shall prove next, an operator on a definite orthosymmetric space has at
most one adjoint.

\begin{proposition}
Let $L,M$ be orthosymmetric spaces with $M$ definite. If $T\in\mathcal{L}%
_{\mathrm{r}}\left(  L,M\right)  $ then $\mathrm{adj}\left(  T\right)  $ has
at most one point.
\end{proposition}

\begin{proof}
Suppose that $R,S\in\mathrm{adj}\left(  T\right)  $. Let $f\in L$ and $g\in
M$. Observe that%
\begin{align*}
\left\langle Rf-Sf,g\right\rangle  &  =\left\langle Rf,g\right\rangle
-\left\langle Sf,g\right\rangle \\
&  =\left\langle f,Tg\right\rangle -\left\langle f,Tg\right\rangle =0.
\end{align*}
Since $g$ are arbitrary in $M$ and $M$ is definite, we get%
\[
Rf-Sf=0\text{ for all }f\in L.
\]
This means that $R=S$ and we are done.
\end{proof}

As it would be expected, if $\mathrm{adj}\left(  T\right)  $ is non-empty,
then its unique element is called the \textsl{adjoint}\textit{ }of $T$ and
denoted by $T^{\ast}$. In this situation, we have%
\[
T^{\ast}\in\mathcal{L}_{\mathrm{r}}\left(  M,L\right)  \text{ and
}\left\langle Tf,g\right\rangle =\left\langle f,T^{\ast}g\right\rangle \text{
for all }f\in L\text{ and }g\in M.
\]
A first property of adjoint operators is given below.

\begin{proposition}
\label{pospos}Let $L,M$ be orthosymmetric spaces with $M$ definite and
$T\in\mathcal{L}_{\mathrm{r}}\left(  L,M\right)  $ such that $T^{\ast}%
\in\mathcal{L}_{\mathrm{r}}\left(  M,L\right)  $ exists. If $T$ is positive
then so is $T^{\ast}$.
\end{proposition}

\begin{proof}
Let $f\in L$ and $g\in M$. Keeping the same notation as previously used in
Section $2$, we can write%
\[
\Phi_{T^{\ast}g}f=\left\langle f,T^{\ast}g\right\rangle =\left\langle
Tf,g\right\rangle =\Phi_{g}\left(  Tf\right)  =\left(  \Phi_{g}\circ T\right)
g.
\]
Hence,%
\[
\Phi_{T^{\ast}g}=\Phi_{g}\circ T\text{ for all }g\in L.
\]
Now, assume that $T$ is positive and let $g\in M^{+}$. We claim that $T^{\ast
}g\in M^{+}$. Since $g\in M^{+}$, the operator $\Phi_{g}$ is positive (see
Lemma \ref{tech} $\mathrm{(i)}$). Hence, $\Phi_{g}\circ T$ is positive and so
is $\Phi_{T^{\ast}g}$. But then%
\[
\left(  T^{\ast}g\right)  ^{-}\in M^{0}=\left\{  0\right\}
\]
as $M$ is definite (where we use Lemma \ref{tech} $\mathrm{(iv)}$). It follows
that $T^{\ast}g\in M^{+}$ which leads to the desired result.
\end{proof}

Now, we shall focus on lattice homomorphisms from $L$ into $M$. Recall that
$T\in\mathcal{L}_{\mathrm{r}}\left(  L,M\right)  $ is a \textsl{lattice
homomorphism}\textit{ }if%
\[
\left\vert Tf\right\vert =T\left\vert f\right\vert \quad\text{for all }f\in
L.
\]
Next, we obtain a (quite amazing) characterization of lattice homomorphisms in
term of adjoint.

\begin{theorem}
\label{riesz}Let $L,M$ be orthosymmetric spaces with $M$ definite and
$T\in\mathcal{L}_{\mathrm{r}}\left(  L,M\right)  $ such that $T$ is positive
and $T^{\ast}\in\mathcal{L}_{\mathrm{r}}\left(  M,L\right)  $ exists. Then $T$
is a lattice homomorphism if and only if $T^{\ast}T\in\mathrm{Orth}\left(
L\right)  $.
\end{theorem}

\begin{proof}
Assume that $T$ is a lattice homomorphism. Since $T$ and $T^{\ast}$ are
positive (see Proposition \ref{pospos}), so is $T^{\ast}T$. Let $f,g\in L$
with $f\wedge g=0$. From $Tf\wedge Tg=0$ it follows that $\left\langle
Tf,Tg\right\rangle =0$. We get%
\begin{align*}
0  &  \leq\left\langle \left(  T^{\ast}T\right)  f\wedge g,\left(  T^{\ast
}T\right)  f\wedge g\right\rangle \leq\left\langle \left(  T^{\ast}T\right)
f,g\right\rangle \\
&  =\left\langle T^{\ast}Tf,g\right\rangle =\left\langle Tf,Tg\right\rangle
=0.
\end{align*}
Since $M$ is definite, we derive that $\left(  T^{\ast}T\right)  f\wedge g=0$
so $T^{\ast}T\in\mathrm{Orth}\left(  L\right)  $.

Conversely, suppose that $T^{\ast}T\in\mathrm{Orth}\left(  L\right)  $ and
pick $f,g\in L$ with $f\wedge g=0$. Hence, $\left(  T^{\ast}T\right)  f\wedge
g=0$ because $T^{\ast}T$ is a positive orthomorphism. Therefore,%
\[
\left\langle Tf,Tg\right\rangle =\left\langle \left(  T^{\ast}T\right)
f,g\right\rangle =0.
\]
Consequently,%
\[
0\leq\left\langle Tf\wedge Tg,Tf\wedge Tg\right\rangle \leq\left\langle
Tf,Tg\right\rangle =0.
\]
We derive that $Tf\wedge Tg=0$ as $M$ is definite. This yields that $T$ is a
lattice homomorphism and completes the proof.
\end{proof}

A quite curious consequences of Theorem \ref{riesz} is discussed next. Let $T$
be a positive $n\times n$ matrix $(n\in\mathbb{N})$ such that $T^{\ast}T$ is
diagonal. Then each row of $T$ contains at most one nonzero (positive) entry.
Indeed, since $T^{\ast}T$ is diagonal, $T^{\ast}T$ is an orthomorphism on
$\mathbb{R}^{n}$. But then $T$ is a lattice homomorphism (where we use Theorem
\ref{riesz}) and the result follows (see \cite{AA02} or \cite{S74}).

We proceed to a question which arises naturally. Namely, is the eventual
adjoint of a lattice homomorphism on orthosymmetric spaces again a lattice
homomorphism ? The following simple example proves that this is not true in general.

\begin{example}
Let $L=\mathbb{R}^{3}$ be equipped with its structure of definite
orthosymmetric space over $\mathbb{R}$ as explained in \emph{Example
\ref{Euclidean}}. Consider%
\[
T=\left(
\begin{array}
[c]{ccc}%
0 & 0 & 0\\
0 & 1 & 0\\
0 & 1 & 0
\end{array}
\right)  \in\mathcal{L}_{\mathrm{r}}\left(  L\right)
\]
and observe that $T$ is a lattice homomorphism on $L$. However,%
\[
T^{\ast}=\left(
\begin{array}
[c]{ccc}%
0 & 0 & 0\\
0 & 1 & 1\\
0 & 0 & 0
\end{array}
\right)
\]
is not a lattice homomorphism.
\end{example}

The situation improves if $T$ is onto as we prove in the following.

\begin{corollary}
Let $L,M$ be orthosymmetric spaces with $M$ definite and $T\in\mathcal{L}%
_{\mathrm{r}}\left(  T\right)  $ be a surjective lattice homomorphism such
that $T^{\ast}\in\mathcal{L}_{\mathrm{r}}\left(  M,L\right)  $ exists. Then
$T^{\ast}$ is again a lattice homomorphism.
\end{corollary}

\begin{proof}
Let $g\in M$ and $f\in L$ such that $g=Tf$. By Theorem \ref{riesz}, the
operator $T^{\ast}T$ is an orthomorphism. Thus,%
\[
T^{\ast}\left\vert g\right\vert =T^{\ast}\left\vert Tf\right\vert =T^{\ast
}T\left\vert f\right\vert =\left\vert T^{\ast}Tf\right\vert =\left\vert
T^{\ast}g\right\vert .
\]
This means that $T^{\ast}$ is a lattice homomorphism, as desired.
\end{proof}

The following shows that only normal lattice homomorphisms have adjoints.
First, recall that a lattice homomorphism $T\in\mathcal{L}_{\mathrm{r}}\left(
L,M\right)  $ is said to be \textsl{normal} if its kernel $\ker\left(
T\right)  $ is a band of $L$.

\begin{corollary}
Let $L,M$ be orthosymmetric space with $M$ definite and $T\in\mathcal{L}%
_{\mathrm{r}}\left(  L,M\right)  $ be a lattice homomorphism such that
$T^{\ast}\in\mathcal{L}_{\mathrm{r}}\left(  M,L\right)  $ exists. Then $T$ is normal.
\end{corollary}

\begin{proof}
We prove that $\ker\left(  T\right)  $ is band. We claim that $\ker\left(
T^{\ast}T\right)  =\ker\left(  T\right)  $. Obviously, $\ker\left(  T^{\ast
}T\right)  $ contains $\ker\left(  T\right)  $. Conversely, if $f\in
\ker\left(  T^{\ast}T\right)  $ then%
\[
0\leq\left\langle Tf,Tf\right\rangle =\left\langle T^{\ast}Tf,f\right\rangle
=0.
\]
It follows that $Tf=0$ because $M$ is definite. We derive that $\ker\left(
T^{\ast}T\right)  =\ker\left(  T\right)  $, as required. On the other hand,
Theorem \ref{riesz} guaranties that $T^{\ast}T$ is an orthomorphism on $L$.
Accordingly, $\ker\left(  T^{\ast}T\right)  $ is a band of $L$ and so is
$\ker\left(  T\right)  $. This yields that $T$ is normal and completes the proof.
\end{proof}

The last part of the paper deals with interval preserving operators. For any
positive element $f$ in a vector lattice $L$, we set%
\[
\left[  0,f\right]  =\left\{  g\in M:0\leq g\leq f\right\}  .
\]
A positive operator $T:M\rightarrow L$ is said to be \textsl{interval
preserving} if%
\[
T\left[  0,f\right]  =\left[  0,Tf\right]  \text{ for all }f\in M^{+}\text{,}%
\]
A sufficient condition for a positive operator acting on orthosymmetric spaces
to be a lattice homomorphism is to have an interval preserving adjoint.

\begin{theorem}
Let $L,M$ be orthosymmetric space with $M$ definite and $T\in\mathcal{L}%
_{\mathrm{r}}\left(  L,M\right)  $ be positive such that $T^{\ast}$ exists. If
$T^{\ast}$ is interval preserving then $T$ is a lattice homomorphism.
\end{theorem}

\begin{proof}
We choose $f\in L$ and we claim that $T\left(  f^{+}\right)  =\left(
Tf\right)  ^{+}$. To this end, observe that if $0\leq h\in M^{+}$ then%
\begin{align*}
\Phi_{\left(  Tf\right)  ^{+}}h  &  =\Phi_{Tf}^{+}h=\sup\left\{  \left\langle
Tf,g\right\rangle :0\leq g\leq h\text{ in }M\right\} \\
&  =\sup\left\{  \left\langle f,T^{\ast}g\right\rangle :0\leq g\leq h\text{ in
}M\right\} \\
&  \leq\sup\left\{  \left\langle f,u\right\rangle :0\leq u\leq T^{\ast}h\text{
in }L\right\}  .
\end{align*}
On the other hand, if $0\leq u\leq T^{\ast}h$ in $L$, then there exists
$g\in\left[  0,h\right]  $ such that $u=T^{\ast}g$. Accordingly,%
\begin{align*}
\Phi_{\left(  Tf\right)  ^{+}}h  &  =\sup\left\{  \left\langle
f,u\right\rangle :0\leq u\leq T^{\ast}h\text{ in }L\right\} \\
&  =\Phi_{f}^{+}T^{\ast}h=\Phi_{f^{+}}T^{\ast}h=\Phi_{T\left(  f^{+}\right)
}h.
\end{align*}
It follows that $\Phi_{\left(  Tf\right)  ^{+}}=\Phi_{T\left(  f^{+}\right)
}$ and so $\left(  Tf\right)  ^{+}=T\left(  f^{+}\right)  $ because $M$ is
definite. This ends the proof of the theorem.
\end{proof}

The converse of the previous theorem fails as it can be seen in the following
example, which is the last item of this paper.

\begin{example}
Let $L=M=C\left[  0,1\right]  $ be equipped with any structure of definite
symmetric space and define $T\in\mathrm{Orth}\left(  L\right)  $ by%
\[
\left(  Tf\right)  \left(  x\right)  =xf\left(  x\right)  \text{ for all }f\in
L\text{ and }x\in\left[  0,1\right]  .
\]
Since $T\in\mathrm{Orth}\left(  L\right)  $, we derive that $T^{\ast}$ exists
and $T=T^{\ast}$ \emph{(}see \emph{Theorem \ref{orth}).} Of course, $T$ is a
lattice homomorphism. However, $T$ is not interval preserving. Indeed, if
$g\in L$ is defined by%
\[
g\left(  x\right)  =x\left\vert \sin\frac{1}{x}\right\vert \text{ if }%
x\in\left]  0,1\right]  \text{ and }g\left(  0\right)  =0,
\]
then $0\leq g\leq T\mathbf{1}$ but there is not $f\in L$ such that $g=Tf$.
\end{example}

\end{document}